\documentclass[11pt,leqno]{amsart}
\usepackage{amssymb,amsmath,amsthm,amsfonts}
\usepackage {latexsym}
\usepackage{bbm}

\allowdisplaybreaks
\newcommand{\nc}{\newcommand}
\nc{\les}{\lesssim}
\nc{\nit}{\noindent}
\nc{\nn}{\nonumber}
\nc{\D}{\partial}
\nc{\diff}[2]{\frac{d #1}{d #2}}
\nc{\diffn}[3]{\frac{d^{#3} #1}{d {#2}^{#3}}}
\nc{\pdiff}[2]{\frac{\partial #1}{\partial #2}}
\nc{\pdiffn}[3]{\frac{\partial^{#3} #1}{\partial{#2}^{#3}}}
\nc{\abs}[1] {\lvert #1 \rvert}
\nc{\cAc}{{\cal A}_c}
\nc{\cE}{{\cal E}}
\nc{\cF}{{\cal F}}
\nc{\cP}{{\cal P}}
\nc{\cV}{{\cal V}}
\nc{\cQ}{{\cal Q}}
\nc{\cGin}{{\cal G}_{\rm in}}
\nc{\cGout}{{\cal G}_{\rm out}}
\nc{\cO}{{\cal O}}
\nc{\Lav}{{\cal L}_{\rm av}}
\nc{\cL}{{\cal L}}
\nc{\cB}{{\cal B}}
\nc{\cZ}{{\cal Z}}
\nc{\cR}{{\cal R}}
\nc{\cT}{{\cal T}}
\nc{\cY}{{\cal Y}}
\nc{\cX}{{\cal X}}
\nc{\cXT}{{{\cal X}(T)}}
\nc{\cBT}{{{\cal B}(T)}}
\nc{\vD}{{\vec \mathcal{D}}}
\nc{\efield}{\mathcal{E}}
\nc{\vE}{{\vec \efield}}
\nc{\vB}{{\vec \mathcal{B}}}
\nc{\vH}{{\vec \mathcal{H}}}
\nc{\ty}{{\tilde y}}
\nc{\tu}{{\tilde u}}
\nc{\tV}{{\tilde V}}
\nc{\Pc}{{\bf P_c}}
\nc{\bx}{{\bf x}}
\nc{\bX}{{\bf X}}
\nc{\bXYZ}{{\bf XYZ}}
\nc{\bY}{{\bf Y}}
\nc{\bF}{{\bf F}}
\nc{\bS}{{\bf S}}
\nc{\dV}{{\delta V}}
\nc{\dE}{{\delta E}}
\nc{\TT}{{\Theta}}
\nc{\dPsi}{{\delta\Psi}}
\nc{\order}{{\cal O}}
\nc{\Rout}{R_{\rm out}}
\nc{\eplus}{e_+}
\nc{\eminus}{e_-}
\nc{\epm}{e_\pm}
\nc{\eps}{\varepsilon}
\nc{\vnabla}{{\vec\nabla}}
\nc{\G}{\Gamma}
\nc{\w}{\omega}
\nc{\mh}{h}
\nc{\mg}{g}
\nc{\vphi}{\varphi}
\nc{\tlambda}{\tilde\lambda}
\nc{\be}{\begin{equation}}
\nc{\ee}{\end{equation}}
\nc{\ba}{\begin{eqnarray}}
\nc{\ea}{\end{eqnarray}}

\nc{\g}{\gamma}
\nc{\ol}{\overline}

\newtheorem{theorem}{Theorem}[section]
\newtheorem{lemma}[theorem]{Lemma}
\newtheorem{prop}[theorem]{Proposition}
\newtheorem{corollary}[theorem]{Corollary}
\newtheorem{defin}[theorem]{Definition}

\nc{\pT}{\partial_T}
\nc{\pz}{\partial_z}
\nc{\pt}{\partial_t}
\nc{\la}{\langle}
\nc{\ra}{\rangle}
\nc{\infint}{\int_{-\infty}^{\infty}}
\nc{\halfwidth}{6.5cm}
\nc{\figwidth}{10cm}
\newcommand{\f}{\frac}

\nc{\nlayers}{L} \nc{\nsectors}{M}
\nc{\indicator}{\mathbf{1}}
\nc{\Rhole}{R_{\rm hole}}
\nc{\Rring}{R_{\rm ring}}
\nc{\neff}{n_{\rm eff}}
\nc{\Frem}{F_{\rm rem}}
\nc{\R}{\mathbb R}
\nc{\Z}{\mathbb Z}
\nc{\DD}{\Delta}
\nc{\cD}{\mathcal D}
\nc{\lnorm}{\left\|}
\nc{\rnorm}{\right\|}
\nc{\rnormp}{\right\|_{\ell^{p,\eps}}}
\nc{\rar}{\rightarrow}
\date{\today}
\begin{document}

\begin{abstract}
Let $H=-\Delta+V$, where $V$ is a real valued potential on $\R^2$ satisfying $|V(x)|\les \la x\ra^{-3-}$.
We prove that if zero is a regular point of the spectrum of $H=-\Delta+V$, then
$$
\|w^{-1} e^{itH}P_{ac}f\|_{L^\infty(\R^2)}\les \f1{|t|\log^2(|t|)} \|w f\|_{L^1(\R^2)},\,\,\,\,\,\,\,\, |t| >2,
$$
with $w(x)=\log^2(2+|x|)$. This decay rate was obtained by Murata in the setting of weighted $L^2$ spaces
with polynomially growing weights.
\end{abstract}

\title[Weighted Dispersive Estimate for  the Schr\"odinger Equation]{\textit{A  weighted dispersive estimate  for Schr\"{o}dinger operators in
dimension two}}

\author[M.~B. Erdo\smash{\u{g}}an, W.~R. Green]{M. Burak Erdo\smash{\u{g}}an and William~R. Green}

\address{Department of Mathematics \\
University of Illinois \\
Urbana, IL 61801, U.S.A.}
\email{berdogan@math.uiuc.edu}
\address{Department of Mathematics and Computer Science\\
Eastern Illinois University \\
Charleston, IL 61920, U.S.A.}
\email{wrgreen2@eiu.edu}

\maketitle
\section{Introduction}
The free Schr\"{o}dinger evolution on $\mathbb R^d$,
$$
 e^{-it\Delta}f (x)= C_d \frac{1}{t^{d/2}}\int_{\R^d} e^{-i|x-y|^2/4t}f(y) dy,
$$
satisfies the dispersive estimate
$$\| e^{-it\Delta}f \|_\infty \lesssim \f1{|t|^{d/2}} \| f\|_1.$$
In recent years many authors (see, e.g.,  \cite{JSS,RodSch,GS, goldberg, Gol2,  Sc2, GV, Yaj, FY, CCV, EG1}, and the survey article \cite{Scs}) worked on the problem of extending this bound to the perturbed Schr\"{o}dinger operator $H=-\Delta+V$, where $V$ is a real-valued potential with sufficient decay at infinity and some smoothness for $d>3$. Since the perturbed operator may have negative point spectrum one needs to consider $e^{itH}P_{ac}(H)$, where $P_{ac}(H)$ is the orthogonal projection onto the absolutely continuous subspace of $L^2(\R^2)$.
One also assumes that zero is a regular point of the spectrum of $H$. This is equivalent to the boundedness of the resolvent,
$$
R_V^\pm(\lambda^2)= R_V(\lambda^2\pm i0)=(H-(\lambda^2 \pm i0))^{-1},
$$
as an operator between certain weighted $L^2$ spaces as $\lambda\to 0$.

It is easy to see that $t^{-d/2}$ decay rate at infinity is optimal for the free evolution. In dimensions $d\geq 3$ one can not hope to have a faster decay rate for the perturbed operator. In fact, it is known that (see, e.g., \cite{Rauch,JenKat,Mur,Jen,Jen2,ES2,Yaj3, goldE, Bec, EG}) the decay rate as $t\to\infty$ is in general  slower if zero is not regular point of the spectrum. In dimensions $d=1$ and $d=2$,  zero is not a regular point of the spectrum of the Laplacian since the constant function is a resonance. Therefore, for the perturbed operator $-\Delta+V$, one may expect to have a faster dispersive decay at infinity if zero is regular. Indeed, in \cite[Theorem 7.6]{Mur}, Murata proved that if zero is a regular point of the spectrum, then for $|t|>2$
\begin{align*}
\| w_1^{-1}e^{itH}P_{ac}(H) f\|_{L^2(\R^1)} &\les |t|^{-\f32} \|w_1 f\|_{L^2(\R^1)},\\
\| w_2^{-1}e^{itH}P_{ac}(H) f\|_{L^2(\R^2)} &\les |t|^{-1}(\log |t|)^{-2}\|w_2 f\|_{L^2(\R^2)}.
\end{align*}
 Here $w_1$ and $w_2$ are weight functions growing at a polynomial rate at infinity. It is also assumed that the potential decays at a polynomial rate at infinity (for $d=2$, it suffices to assume that $w_2(x)=\la x\ra^{-3-}$ and  $|V(x)|\les \la x\ra^{-6-}$ where $\la x\ra:= (1+|x|^2)^\f12$).
This type of estimates are very useful in the study of nonlinear asymptotic stability of (multi) solitons in lower dimensions since the dispersive decay rate  in time is integrable at infinity (see, e.g., \cite{ scns, KZ, Miz}). Also see \cite{BP, SW, PW, Wed} for other applications of weighted dispersive estimates to nonlinear PDEs.

In \cite{Scs}, 
Schlag extended Murata's result for $d=1$ to the $L^1\to L^\infty$ setting. He proved that if zero is regular, then
\begin{align*}
	\|w^{-1} e^{itH}P_{ac}(H) f\|_{L^\infty(\R)}\les |t|^{-\frac{3}{2}}\|w f\|_1, \,\,\,\,\,\,\,\, |t|>2,
\end{align*}
with $w(x)=\la x\ra$ provided $\|\la x\ra^4 V\|_1<\infty$.

In this paper, we study the two dimensional case. Our main result is the following
\begin{theorem}\label{thm:main}
Let $V(x)\les \la x\ra^{-2\beta}$ for some $\beta>\f32$. If zero is a regular point of the spectrum of $H=-\Delta+V$, then we have
$$
\|w^{-1} e^{itH}P_{ac}f\|_{L^\infty(\R^2)}\les \f1{|t| \log^2(|t|)} \|w f\|_{L^1(\R^2)},\,\,\,\,\,\,\,\,|t|>2,
$$
where $w(x)=\log^2(2+|x|)$.
\end{theorem}
We note that the requirement for the weight function and the potential is much weaker than  was assumed in \cite{Mur}. We think similar bounds hold in the case of matrix Schr\"odinger operators, which we plan to address in a subsequent paper.

There are not many results on $L^1 \to L^\infty$ estimates in the two dimensional case.
In \cite{Sc2}, Schlag proved that
$$\|e^{itH}P_{ac}\|_{L^1(\R^2)\to L^\infty(\R^2)}\lesssim |t|^{-1}$$
under the decay assumption $|V|\lesssim \langle x\rangle^{-3-}$ and the assumption that zero is a regular point of the spectrum.  For the case when zero is not regular, see \cite{EG}. Yajima, \cite{Yaj2}, established that the wave operators are bounded on $L^p(\R^2)$ for $1<p<\infty$ if zero is regular.  The hypotheses
on the potential $V$ were
relaxed slightly in \cite{JenYaj}. High frequency dispersive estimates, similar to those obtained in \cite{Sc2}
were obtained by Moulin, \cite{Mou}, under an integrability condition on the potential.

We also note that standard spectral theoretic results for $H$ apply.  Under our  assumptions  we have that the spectrum of $H$ can be expressed as the absolutely continuous spectrum, the interval
$[0,\infty)$, and finitely many eigenvalues of finite multiplicity on $(-\infty,0]$.  See
\cite{RS1} for spectral theory and \cite{Stoi} for Birman-Schwinger type
bounds.

As usual, Theorem~\ref{thm:main} follows from (see, e.g., \cite{GS, Sc2, EG})
\be\label{mainineq}
	\sup_{L\geq1}\bigg| \int_0^\infty e^{it\lambda^2}\lambda \chi(\lambda/L)
	[R_V^+(\lambda^2)-R_V^-(\lambda^2)](x,y) d\lambda \bigg|
	\les \f{w(x)w(y)}{t \log^2(t)},\,\,\,\,\,\,\,\,t>2.
\ee
Here $\chi$ is an even smooth function supported in $[-\lambda_1,\lambda_1]$ and $\chi(x)=1$ for $|x|<\lambda_1/2$, and $\lambda_1$ is a sufficiently small number which is fixed throughout the paper.

In this paper we prove that
\begin{theorem} \label{thm:mainineq}
	Under the assumptions of Theorem~\ref{thm:main}, we have
	for $t>2$
	\be\label{weighteddecay}
		\sup_{L\geq1}\bigg| \int_0^\infty e^{it\lambda^2}\lambda \chi(\lambda/L)
		 [R_V^+(\lambda^2)-R_V^-(\lambda^2)](x,y) d\lambda \bigg|
		 \les \f{\sqrt{w(x)w(y)}}{t\log^2(t)}+
		\f{\la x\ra^{\f32} \la y\ra^{\f32}}{t^{1+\alpha}},
	\ee
	where $0<\alpha<\min(\f14, \beta-\f32)$.
\end{theorem}
Our proof of Theorem~\ref{thm:mainineq} will be mostly self-contained. Since we can allow polynomial growth in $x$ and $y$ for many terms that arise, the proof is somehow less technical than the proof   in \cite{Sc2}.

To obtain Theorem~\ref{thm:main}, we use the inequality
$$
\min\big(1,\f{a}{b}\big)\les \f{\log^2(a)}{\log^2(b)},\,\,\,\,\,\,\,a,b>2,
$$
and interpolate \eqref{weighteddecay} with the result of Schlag in \cite{Sc2}, which states that under the conditions of Theorem~\ref{thm:main}, one has
\be\nn
	\sup_{L\geq1} \bigg| \int_0^\infty e^{it\lambda^2}\lambda \chi(\lambda/L)
	[R_V^+(\lambda^2)-R_V^-(\lambda^2)](x,y) d\lambda \bigg| \les \frac{1}{t}.
\ee

\section{The Free Resolvent} \label{sec:exp}

In this section we discuss the properties of the free resolvent,  $R_0^\pm(\lambda^2)=[-\Delta-(\lambda^2\pm i 0)]^{-1}$, in $\R^2$.
To simplify the formulas, we   use the notation
$$
f=\widetilde O(g)
$$
to denote
$$
\frac{d^j}{d\lambda^j} f = O\big(\frac{d^j}{d\lambda^j} g\big),\,\,\,\,\,j=0,1,2,3,...
$$
If the derivative bounds hold only for the first $k$ derivatives we  write $f=\widetilde O_k (g)$.

Recall that
\begin{align}\label{R0 def}
	R_0^\pm(\lambda^2)(x,y)=\pm\frac{i}{4} H_0^\pm(\lambda|x-y|)=\pm\frac{i}{4} J_0(z)- \frac{1}{4} Y_0(z).
\end{align}
Thus, we have
\be\label{r0low2}
  R_0^+(\lambda^2)(x,y)-R_0^-(\lambda^2)(x,y) =  \frac{i}2 J_0(\lambda |x-y|).
\ee
From the series expansions for the Bessel functions, see \cite{AS},  we have, as $z\to 0$,
\begin{align}
	J_0(z)&=1-\frac{1}{4}z^2+\frac{1}{64}z^4+\widetilde O_6(z^6),\label{J0 def}\\
	Y_0(z)&=\frac{2}{\pi}(\log(z/2)+\gamma)J_0(z)+\frac{2}{\pi}\bigg(\frac{1}{4}z^2 +\widetilde O_4(z^4)\bigg)\nn \\
&=\frac2\pi \log(z/2)+\f{2\gamma}{\pi}+ \widetilde O(z^2\log(z)).\label{Y0 def}
\end{align}
Further, for $|z|>1 $, we have the representation (see, {\em e.g.}, \cite{AS})
\begin{align}\label{JYasymp2}
	 H_0^\pm(z)= e^{\pm i z} \omega_\pm(z),\,\,\,\,\quad \omega_{\pm}(z)=\widetilde O\big((1+|z|)^{-\frac{1}{2}}\big).
\end{align}
This implies that for $|z|>1$
\begin{align}\label{largeJYH}
	\mathcal C(z)=e^{iz} \omega_+(z)+e^{-iz}\omega_-(z), \qquad
	\omega_{\pm}(z)=\widetilde O\big((1+|z|)^{-\frac{1}{2}}\big),
\end{align}
 for any $\mathcal C\in\{J_0, Y_0\}$ respectively with different $\omega_{\pm}$.

In particular, for $\lambda|x-y|\les 1$, we have
\be\label{r0low}
 R_0^\pm(\lambda^2)(x,y)= \pm\frac{i}4-\frac{\gamma}{2\pi}-\frac{1}{2\pi}\log(\lambda|x-y|/2)+\widetilde O\big(\lambda^2|x-y|^2\log(\lambda|x-y|)\big).
\ee
For $\lambda|x-y|\gtrsim 1$, we have
\be\label{r0high}
 R_0^\pm(\lambda^2)(x,y) = e^{i\lambda|x-y|}\omega_+(\lambda|x-y|)+ e^{-i\lambda|x-y|}\omega_-(\lambda|x-y|).
\ee

\section{Resolvent Expansion Around the Zero Energy}
Let $U(x)=1$ if $V(x)\geq 0$ and $U(x)=-1$ if $V(x)<0$, and let $v=|V|^{1/2}$. We have $V=Uv^2$.
We use the symmetric resolvent identity, valid for $\Im\lambda>0$:
\be\label{res_exp}
R_V^\pm(\lambda^2)=  R_0^\pm(\lambda^2)-R_0^\pm(\lambda^2)vM^\pm(\lambda)^{-1}vR_0^\pm(\lambda^2),
\ee
where $M^\pm(\lambda)=U+vR_0^\pm(\lambda^2)v$.
The key issue in the resolvent expansions is the invertibility of the operator $M^\pm(\lambda)$ for small $\lambda$
under various spectral assumptions at zero. Below, using the properties of the free resolvent listed
above, we provide an expansion for the free resolvent around $\lambda=0$, and then using it obtain analogous expansions of the operator $M^\pm(\lambda)$. Similar lemmas were proved in \cite{JN} and \cite{Sc2}, however we need to  obtain slightly different error bounds. The following operator and the function arise naturally in the expansion of $M^{\pm}(\lambda)$
(see \eqref{Y0 def})
\begin{align}
	G_0f(x)&=-\frac{1}{2\pi}\int_{\R^2} \log|x-y|f(y)\,dy, \label{G0 def}\\
\label{g form}
		g^{\pm}(\lambda)&:= \|V\|_1\Big(\pm \frac{i}{4}-\frac{1}{2\pi}\log(\lambda/2)-\frac{\gamma}{2\pi} \Big).
\end{align}
\begin{lemma}\label{R0 exp cor}	
	We have the following expansion for the kernel of the free resolvent
	\begin{align*}
		R_0^{\pm}(\lambda^2)(x,y)=\frac{1}{\|V\|_1} g^{\pm}(\lambda)+G_0(x,y)
		+E_0^{\pm}(\lambda)(x,y).
	\end{align*}
	Here $G_0(x,y)$ is the kernel of the operator $G_0$ in \eqref{G0 def}, $g^{\pm}(\lambda)$ is as in \eqref{g form}, and
$E_0^{\pm}$
	satisfies the bounds
	\begin{align*}
		|E_0^{\pm}|\les \lambda^{\frac{1}{2} }|x-y|^{\frac{1}{2} }, \qquad
		|\partial_\lambda E_0^{\pm}|\les \lambda^{-\frac{1}{2} }|x-y|^{\frac{1}{2} }, \qquad|\partial_\lambda^2 E_0^{\pm}|\les \lambda^{-\frac{1}{2} }|x-y|^{\frac{3}{2}}.
	\end{align*}
\end{lemma}

\begin{proof}
	To obtain the expansions recall \eqref{r0low}, which states that for $\lambda|x-y|\les 1$, we have
\begin{multline*}
 R_0^\pm(\lambda^2)(x,y)    = \pm\frac{i}4-\frac{\gamma}{2\pi}-\frac{1}{2\pi}\log(\lambda|x-y|/2)+\widetilde O\big(\lambda^2|x-y|^2\log(\lambda|x-y|)\big) \\
=  \f{g^{\pm}(\lambda)}{ \|V\|_1}+G_0(x,y)+\widetilde O\big(\lambda^2|x-y|^2\log(\lambda|x-y|)\big).
\end{multline*}
For $\lambda|x-y|\gtrsim 1$, using \eqref{r0high} we have
\begin{multline*}
 R_0^\pm(\lambda^2)(x,y)    = e^{i\lambda|x-y|}\omega_+(\lambda|x-y|)+ e^{-i\lambda|x-y|}\omega_-(\lambda|x-y|) \\
=  \f{g^{\pm}(\lambda)}{ \|V\|_1}+G_0(x,y) +\widetilde O\big(\log(\lambda|x-y|)\big)+ \widetilde O\big(\f{e^{i\lambda|x-y|}}{(1+\lambda|x-y|)^{1/2}}\big).
\end{multline*}
Let $\chi$ be a smooth cutoff for $[-1,1]$, and $\widetilde \chi=1-\chi$. Using the formulas above we have
\begin{align*}
		 E_0^{\pm}(\lambda)(x,y) \chi(\lambda|x-y|) & =   \chi(\lambda|x-y|) \widetilde O\big(\lambda|x-y|)^2\log(\lambda|x-y|)\big), \\
 E_0^{\pm}(\lambda)(x,y) \widetilde\chi(\lambda|x-y|) & =    \widetilde\chi(\lambda|x-y|)
\big[\widetilde O\big(\log(\lambda|x-y|\big)+ \widetilde O\big(\f{e^{i\lambda|x-y|}}{(1+\lambda|x-y|)^{1/2}}\big)\big].
\end{align*}
Combining these bounds we have
$$
		|E_0^{\pm}(\lambda)(x,y)| \lesssim  \big[(\lambda|x-y|)^{2-}\chi(\lambda|x-y|)
		+ (\lambda|x-y|)^{0+}\widetilde\chi(\lambda|x-y|) \big]    \lesssim  (\lambda|x-y|)^{\frac{1}{2} }.
$$
For $\lambda$-derivatives, note that
$$
		|\partial_\lambda E_0^{\pm}(\lambda)(x,y)|\lesssim   \f{(\lambda |x-y|)^{2-}}{\lambda} \chi(\lambda|x-y|)+  \frac{|x-y|^{1/2}}{\lambda^{1/2}}
		\widetilde\chi(\lambda|x-y|)    \lesssim  \lambda^{-\frac{1}{2} } |x-y|^{\frac{1}{2} },
$$
and
$$
		|\partial_\lambda^2 E_0^\pm(\lambda)(x,y)|\lesssim  \f{(\lambda |x-y|)^{2-}}{\lambda^2} \chi(\lambda|x-y|)+  \frac{|x-y|^{3/2}}{\lambda^{1/2}} \widetilde\chi(\lambda|x-y|)   \lesssim   \lambda^{-\frac{1}{2} } |x-y|^{\frac{3}{2} }.
$$
\end{proof}
The following corollary follows from the bounds for $\partial_\lambda E_0^\pm$ and $\partial^2_\lambda E_0^\pm$.
\begin{corollary} \label{lipbound} For $0<\alpha<1$  and $b>a>0$ we have
$$
|\partial_\lambda E_0^\pm(b)-\partial_\lambda E_0^\pm(a)|\les a^{-\f12} |b-a|^{\alpha} |x-y|^{\frac12+\alpha}.
$$
\end{corollary}
\begin{proof}
 The mean value theorem together with the bound on $\partial^2_\lambda E_0^\pm$ from Lemma~\ref{R0 exp cor} imply that
$$
|\partial_\lambda E_0^\pm(b)-\partial_\lambda E_0^\pm(a)|\les   a^{-1/2} |b-a|  |x-y|^{\frac32}.
$$
Interpolating this with the bound on $\partial_\lambda E_0^\pm$ from Lemma~\ref{R0 exp cor} yields the claim.
\end{proof}

\begin{lemma} \label{lem:M_exp} Let $0<\alpha<1$.
	For $\lambda>0$ define $M^\pm(\lambda):=U+vR_0^\pm(\lambda^2)v$.
	Let $P=v\langle \cdot, v\rangle \|V\|_1^{-1}$ denote the orthogonal projection onto $v$.  Then
	\begin{align*}
		M^{\pm}(\lambda)=g^{\pm}(\lambda)P+T+E_1^{\pm}(\lambda).
	\end{align*}
	Here
$T=U+vG_0v$ where $G_0$ is an
	integral operator defined in \eqref{G0 def}.
	Further, the error term satisfies the bound
	\begin{multline*}
		\big\| \sup_{0<\lambda<\lambda_1} \lambda^{-\frac{1}{2}} |E_1^{\pm}(\lambda)|\big\|_{HS}
		+\big\| \sup_{0<\lambda<\lambda_1} \lambda^{\frac{1}{2}} |\partial_\lambda E_1^{\pm}(\lambda)|\big\|_{HS}	
		\\+\big\| \sup_{0<\lambda<b<\lambda_1} \lambda^{\frac{1}{2}} (b-\lambda)^{-\alpha} |\partial_\lambda E_1^{\pm}(b)-\partial_\lambda E_1^\pm(\lambda)|\big\|_{HS}	
		\les 1,
	\end{multline*}
	provided that $v(x)\lesssim \langle x\rangle^{-\frac{3}{2}-\alpha-}$.
	
\end{lemma}

\begin{proof} Note that
$$
		E_1^{\pm}(\lambda) =M^{\pm}(\lambda)-[g^{\pm}(\lambda)P+T]=  vR_0^{\pm}(\lambda^2)v - g^{\pm}(\lambda)P-vG_0v =vE_0^\pm(\lambda)v.
$$
Therefore the statement follows from Lemma~\ref{R0 exp cor} and Corollary~\ref{lipbound}, and the fact that for $k\geq 0$,
$v(x)|x-y|^k v(y)$ is Hilbert-Schmidt on $L^{2}(\R^2)$ provided that $v(x)\les \la x\ra^{-k-1-}$.
\end{proof}

Recall the following  definition from \cite{Sc2} and \cite{EG}.
\begin{defin}
	We say an operator $T:L^2(\R^2)\to L^2(\R^2)$ with kernel
	$T(\cdot,\cdot)$ is absolutely bounded if the operator with kernel
	$|T(\cdot,\cdot)|$ is bounded from $L^2(\R^2)$ to $L^2(\R^2)$.
\end{defin}
It is worth noting that finite rank operators and  Hilbert-Schmidt operators are absolutely bounded.
Also recall the following definition from \cite{JN}, also see \cite{Sc2} and \cite{EG}.
\begin{defin}\label{resondef}
Let $Q:=\mathbbm{1}-P$.
We say zero is a regular point of the spectrum
of $H = -\Delta+ V$ provided $ QTQ=Q(U + vG_0v)Q$ is invertible on $QL^2(\mathbb R^2)$.
\end{defin}
In \cite{Sc2}, it was proved that if zero is regular, then the operator $D_0:=(QTQ)^{-1}$ is absolutely bounded on $QL^2$.

Below, we discuss the   invertibility of  $M^\pm(\lambda)=U+vR_0^\pm(\lambda^2)v$, for small $\lambda$. This lemma was proved in  \cite{JN} and in \cite{Sc2}. We include the proof for completeness since we state slightly different error bounds.

\begin{lemma}\label{Minverse}
    Let $0<\alpha<1$.
    Suppose that zero is a regular point of the spectrum of  $H=-\Delta+V$. Then for   sufficiently small $\lambda_1>0$, the operators
	$M^{\pm}(\lambda)$ are invertible for all $0<\lambda<\lambda_1$ as bounded operators on $L^2(\R^2)$.
	Further, one has
	\begin{align}
	\label{M size}
        	 M^{\pm}(\lambda)^{-1}=h_{\pm}(\lambda)^{-1}S+QD_0Q+ E^{\pm}(\lambda),
	\end{align}
	Here
   	$h_\pm(\lambda)=g^\pm(\lambda)+c$  (with $c\in\R$), and
  	\be\label{S_defn}
  	 	 S=\left[\begin{array}{cc} P & -PTQD_0Q\\ -QD_0QTP & QD_0QTPTQD_0Q
		\end{array}\right]
  	\ee
	is a finite-rank operator with real-valued kernel.  Further, the error term satisfies the bounds
	\begin{multline*}
		\big\| \sup_{0<\lambda<\lambda_1} \lambda^{-\frac{1}{2} } |E^{\pm}(\lambda)|\big\|_{HS}
		+\big\| \sup_{0<\lambda<\lambda_1} \lambda^{\frac{1}{2} } |\partial_\lambda E^{\pm}(\lambda)|\big\|_{HS}	\\
		+\big\| \sup_{0<\lambda<b\les\lambda<\lambda_1}  \lambda^{\frac{1}{2}+\alpha} (b-\lambda)^{-\alpha} |\partial_\lambda E^{\pm}(b)-\partial_\lambda E^\pm(a)| \big\|_{HS}	
		\les 1,
	\end{multline*}
	provided that $v(x)\lesssim \langle x\rangle^{-\frac{3}{2}-\alpha-}$.

\end{lemma}

\begin{proof}

	  We will give the proof for $M^+$ and drop the superscript ``$+$" from formulas.  Using Lemma~\ref{lem:M_exp}, we
	  write $M(\lambda)$ with respect to the decomposition
	$L^2(\R^2)=PL^2(\R^2)\oplus QL^2(\R^2)$.
	\begin{align*}
		M(\lambda)=\left[ \begin{array}{cc} g(\lambda)P+PTP & PTQ\\
		QTP & QTQ
		\end{array}\right]+E_1(\lambda).
	\end{align*}
	Denote the matrix component of the above equation by $A(\lambda)=\{a_{ij}(\lambda)\}_{i,j=1}^{2}$.

    Since  $QTQ$ is invertible by the assumption that zero is regular, by the Fehsbach formula invertibility of
    $A(\lambda)$ hinges upon the existence
    of $d=(a_{11}-a_{12}a_{22}^{-1}a_{21})^{-1}$. Denoting $D_0=(QTQ)^{-1}:QL^2\to QL^2$, we have
	\begin{align*}
		a_{11}-a_{12}a_{22}^{-1}a_{21}= g(\lambda)P+PTP-PTQD_0QTP  =h(\lambda) P
	\end{align*}
	with $h(\lambda)=g(\lambda)+Tr(PTP-PTQD_0QTP)=g(\lambda)+c$, where $c\in\R$ as the kernels of $T$, $QD_0Q$ and $v$ are real-valued. The invertibility of this operator on $PL^2$ for small $\lambda$ follows from \eqref{g form}.
Thus, by the
	Fehsbach formula,
	\begin{align}\nonumber
		A(\lambda)^{-1}&=\left[\begin{array}{cc} d & -da_{12}a_{22}^{-1}\\
		-a_{22}^{-1}a_{21}d & a_{22}^{-1}a_{21}da_{12}a_{22}^{-1}+a_{22}^{-1}
		\end{array}\right]\\
		&=h^{-1}(\lambda)\left[\begin{array}{cc} P & -PTQD_0Q\\ -QD_0QTP & QD_0QTPTQD_0Q
		\end{array}\right]+QD_0Q =: h^{-1}(\lambda)S+QD_0Q. \label{Ainverse}
	\end{align}
    Note that $S$ has rank at most two. This and the absolute boundedness of $QD_0Q$ imply that $A^{-1}$
    is absolutely bounded.

    	Finally, we write
    	$$
    		M(\lambda)=A(\lambda)+E_1(\lambda)=[\mathbbm{1}+E_1(\lambda) A^{-1}(\lambda)] A(\lambda).
    	$$
	Therefore, by a Neumann series expansion, we have
	\be\label{M plus S}
        M^{-1}(\lambda) =A^{-1}(\lambda)
        \big[\mathbbm{1}+E_1(\lambda) A^{-1}(\lambda)\big]^{-1}=h(\lambda)^{-1}S
        +QD_0Q+E(\lambda),
	\ee
  	The error bounds follow in light of the bounds for $E_1(\lambda)$ in Lemma~\ref{lem:M_exp} and the fact that, as an absolutely bound operator on $L^2$, $|A^{-1}(\lambda)|\les 1$, $|\partial_\lambda  A^{-1}(\lambda)|\les \lambda^{-1}$, and (for $0<\lambda<b<\lambda_1$)
  $$|\partial_\lambda  A^{-1}(\lambda)-\partial_\lambda  A^{-1}(b)|\les (b-\lambda)^\alpha \lambda^{-1-\alpha}.$$
  In the Lipschitz estimate, the factor $\lambda^{-\f12-\alpha}$  arises from the case when the derivative hits $A^{-1}(\lambda)$.	
\end{proof}

\noindent
{\bf Remark.}  Under the conditions of Theorem~\ref{thm:main}, the resolvent identity
	\begin{multline}
	    R_V^{\pm}(\lambda^2)=R_0^{\pm}(\lambda^2)-R_0^{\pm}(\lambda^2) v M^{\pm}(\lambda)^{-1}v
	    R_0^{\pm}(\lambda^2) \\ \label{resolvent id}
=R_0^{\pm}(\lambda^2)-R_0^{\pm}(\lambda^2) \frac{v S v}{h_\pm(\lambda)}
	    R_0^{\pm}(\lambda^2)
-R_0^{\pm}(\lambda^2)  v QD_0Q v
	    R_0^{\pm}(\lambda^2) - R_0^{\pm}(\lambda^2)  v E^\pm(\lambda) v
	    R_0^{\pm}(\lambda^2)
	\end{multline}
	holds as an  operator identity  between the spaces $L^{2,\frac{1}{2}+}(\R^2)$ and $ L^{2,-\frac{1}{2}-}(\R^2)$, as in the
	limiting absorption principle, \cite{agmon}.

We complete this section by noting that for fixed $x,y$ the kernel $R_V^\pm(\lambda^2)(x,y)$ of the  resolvent  remains bounded as $\lambda \to 0$. This is because of a cancellation between the first and second summands of the second line in \eqref{resolvent id}. A consequence of this cancellation will be crucial in the next section, see Proposition~\ref{freeevol} and Proposition~\ref{prop:stone2_2}.

\section{Proof of Theorem~\ref{thm:mainineq} for Low Energies}\label{sec:low energy}
In this section we prove Theorem~\ref{thm:mainineq} for low energies. Let $\chi$ be a smooth cut-off for $[0,\lambda_1]$ as in the introduction, where $\lambda_1$ is sufficiently small so that the expansions in the previous section are valid.
We have
\begin{theorem}\label{lowprop} Fix $0<\alpha<1/4$. Let $v(x)\lesssim \la x\ra^{-\f32-\alpha-}.$
For any $t > 2$, we have
\be\label{stone2}
\Big|\int_0^\infty e^{it\lambda^2}\lambda \chi(\lambda)  [R_V^+(\lambda^2)-R_V^-(\lambda^2)](x,y)
d\lambda\Big| \les \frac{\sqrt{w(x)w(y)}}{t\log^2(t)}+\f{\la x\ra^{\f32 } \la y\ra^{\f32 }}{t^{1+\alpha}}.
\ee
\end{theorem}

We start with a simple lemma:
\begin{lemma} \label{lem:ibp} For $t>2$, we have
$$
\Big|\int_0^\infty e^{it\lambda^2} \lambda \, \mathcal E(\lambda) d\lambda -\f{i\mathcal E(0)}{2t}\Big| \les \f1t\int_0^{t^{-1/2}}|\mathcal E^\prime(\lambda)| d\lambda+ \Big|\frac{\mathcal{E}^\prime(t^{-1/2})}{t^{3/2}}\Big|
+\f1{t^2}\int_{t^{-1/2}}^\infty \Big|\Big(\frac{\mathcal E^\prime(\lambda)}{\lambda}\Big)^\prime\Big| d\lambda.
$$
\end{lemma}
\begin{proof}
To prove this lemma we  integrate by parts once using the identity $e^{it\lambda^2}\lambda=\partial_\lambda e^{it\lambda^2}/(2it)$, and then divide the integral into pieces on the sets $[0,t^{-1/2}]$ and $[t^{-1/2},\infty)$. Finally integrate by parts once more in the latter piece:
\begin{multline*}
 \int_0^\infty e^{it\lambda^2} \lambda \mathcal E(\lambda) d\lambda = \f{i\mathcal E(0)}{2t} +  \f{i}{2t}  \int_0^{t^{-1/2}} e^{it\lambda^2}  \mathcal E^\prime(\lambda) d\lambda+\f{i}{2t}  \int_{t^{-1/2}}^\infty e^{it\lambda^2} \lambda   \f{\mathcal E^\prime(\lambda)}{\lambda} d\lambda\\
=  \f{i\mathcal E(0)}{2t} +  \f{i}{2t}  \int_0^{t^{-1/2}} e^{it\lambda^2}  \mathcal E^\prime(\lambda) d\lambda-\f1{4t^2}\frac{ \mathcal E^\prime(\lambda)}{\lambda}\Big|_{\lambda=t^{-1/2}} -\f1{4t^2}
\int_{t^{-1/2}}^\infty e^{it\lambda^2}  \Big(\f{ \mathcal E^\prime(\lambda)}{\lambda}\Big)^\prime d\lambda.
\end{multline*}

\end{proof}

We start with the contribution of the free resolvent to \eqref{stone2}. Note that it is easy to obtain this statement for the free evolution using its convolution kernel. We choose to present the proof below to introduce some of the methods we will employ throughout the paper.
\begin{prop}\label{freeevol} We have
$$\int_0^\infty e^{it\lambda^2}\lambda \chi(\lambda)  [R_0^+(\lambda^2)-R_0^-(\lambda^2)](x,y)
d\lambda=-\frac1{4t} +O\Big(\f{\la x \ra^{\f32 } \la y\ra^{\f32 }}{t^{\f54}}\Big).
$$
\end{prop}
\begin{proof}
Using Lemma~\ref{R0 exp cor}, we have
$$
R_0^+-R_0^-=\f{i}{2}+E_0^+(\lambda)-E_0^-(\lambda).
$$
Therefore we can rewrite the $\lambda$ integral above as
\be\nn
\frac{i}{2}\int_0^\infty e^{it\lambda^2}\lambda \chi(\lambda)
d\lambda + \int_0^\infty e^{it\lambda^2} \lambda  \chi(\lambda) (E_0^+(\lambda)-E_0^-(\lambda)) d\lambda=:A+B.
\ee
Note that by integrating by parts twice as in the proof of Lemma~\ref{lem:ibp} we obtain
\be\label{Aest}
A=-\frac{1}{4t}-\frac{i}{8t^2}\int_0^\infty e^{it\lambda^2}\frac{d}{d\lambda}\Big(\f{\chi^\prime(\lambda)}{\lambda}\Big)d\lambda\Big)=-\frac{1}{4t}+O(t^{-2}).
\ee
Using the bounds in Lemma~\ref{R0 exp cor} for $\mathcal E(\lambda)= \chi(\lambda) (E_0^+(\lambda)-E_0^-(\lambda)) $, we see that $\mathcal E(0)=0$, and
\begin{align*}
|\partial_\lambda \mathcal E(\lambda)| &\les   \lambda^{-\f12}|x-y|^{\f12 } \les \lambda^{-\f12 } \sqrt{\la x \ra \la y\ra},\\
 \Big|\partial_\lambda\Big(\frac{\partial_\lambda \mathcal E (\lambda)}{\lambda}\Big)\Big|& \les \chi(\lambda) [\lambda^{-\f52 }|x-y|^{\f12 }+\lambda^{-\f32 }|x-y|^{\f32 }]\les   \lambda^{-\f52 } \la x \ra^{\f32 } \la y\ra^{\f32 }.
 \end{align*}
Applying Lemma~\ref{lem:ibp} with these bounds we obtain
\begin{multline*}
|B|\les \f1t\int_0^{t^{-1/2}}|\mathcal E^\prime(\lambda)| d\lambda+ \Big|\frac{\mathcal{E}^\prime(t^{-1/2})}{t^{3/2}}\Big|
+\f1{t^2}\int_{t^{-1/2}}^\infty \Big|\Big(\frac{\mathcal E^\prime(\lambda)}{\lambda}\Big)^\prime\Big| d\lambda\\
\les \f{\sqrt{ \la x \ra \la y\ra}}{t}\int_0^{t^{-1/2}} \lambda^{-\f12} d\lambda +\f{ \sqrt{ \la x \ra \la y\ra} }{t^{\f54}}
+\f{\la x \ra^{\f32 } \la y\ra^{\f32 }}{t^2}\int_{t^{-1/2}}^\infty  \lambda^{-\f52 }  d\lambda
\les\f{\la x \ra^{\f32 } \la y\ra^{\f32 }}{t^{\f54}}.
\end{multline*}

\end{proof}

We now consider the contribution of the second term in \eqref{resolvent id} to \eqref{stone2}:
\be\label{stone2_2}
 \int_{\R^4}\int_0^\infty e^{it\lambda^2}\lambda \chi(\lambda)  [\mathcal R^- -\mathcal R^+] v(x_1)S(x_1,y_1)v(y_1)
d\lambda dx_1 dy_1,
\ee
where
\be\label{calR}
\mathcal R^\pm=\f{R_0^{\pm}(\lambda^2)(x,x_1) R_0^{\pm}(\lambda^2)(y_1,y)}{h_\pm(\lambda)}.
\ee
\begin{prop}\label{prop:stone2_2} Let $0<\alpha<1/4$. If $v(x)\les \la x\ra^{-\f32-\alpha-}$, then
we have
$$
\eqref{stone2_2}=\frac{1}{4t} + O\Big( \f{\sqrt{w(x)w(y)}  }{t\log^2(t)} \Big)+O\Big( \f{\la x\ra^{\f12+\alpha+} \la y\ra^{\f12+\alpha+}}{t^{1+\alpha} } \Big).
$$
\end{prop}
\begin{proof}
Recall from  Lemma~\ref{R0 exp cor}  that
	\begin{align*}
		R_0^{\pm}(\lambda^2)(x,x_1)=\frac{1}{\|V\|_1} g^{\pm}(\lambda)+G_0(x,x_1)+E_0^\pm(\lambda)(x,x_1).
	\end{align*}
Also recall that $h^\pm(\lambda)=g^\pm(\lambda)+c$ with $c\in \R$. Therefore
$$
\mathcal R^\pm=\frac{1}{\|V\|_1^2}\big[g^\pm(\lambda)+c+\widetilde G_0(x,x_1)+\widetilde G_0(y,y_1) +\frac{\widetilde G_0(x,x_1) \widetilde G_0(y,y_1)}{g^\pm(\lambda)+c}\big] +E_2^\pm(\lambda),
$$
where
\begin{multline}\label{E2def}
E_2^\pm(\lambda):=\f1{\|V\|_1}\Big(1+\f{\widetilde G_0(x,x_1)}{g^\pm(\lambda)+c}\Big) E_0^\pm(\lambda)(y,y_1) +
\f1{\|V\|_1}\Big(1+\f{\widetilde G_0(y,y_1)}{g^\pm(\lambda)+c}\Big) E_0^\pm(\lambda)(x,x_1)\\+\f{E_0^\pm(\lambda)(x,x_1)E_0^\pm(\lambda)(y,y_1)}{g^\pm(\lambda)+c},
\end{multline}
and $\widetilde G_0=\|V\|_1G_0-c$. Using this and  \eqref{g form}, we have
$$
\mathcal R^--\mathcal R^+=-\frac{i}{2\|V\|_1}+c_3\frac{\widetilde G_0(x,x_1) \widetilde G_0(y,y_1)}{(\log(\lambda)+c_1)^2+c_2^2} +E_2^-(\lambda)-E^+_2(\lambda),
$$
where  $c_1,c_2,c_3\in\R$.

Accordingly we rewrite the $\lambda$--integral in \eqref{stone2_2} as a sum of the following
\begin{align}\label{Sll1}
&-\frac{i}{2\|V\|_1}  \int_0^\infty e^{it\lambda^2} \lambda \chi(\lambda) d\lambda,
\\&\label{Sll2}
\int_0^\infty e^{it\lambda^2} \lambda \chi(\lambda) \frac{\widetilde G_0(x,x_1) \widetilde G_0(y,y_1)}{(\log(\lambda)+c_1)^2+c_2^2}   d\lambda,
\\&\label{Sll3}
\int_0^\infty e^{it\lambda^2} \lambda \chi(\lambda) [E_2^-(\lambda)-E^+_2(\lambda)] d\lambda.
\end{align}
Note that by \eqref{Aest} we have
\be\label{lem:Sll1}
\eqref{Sll1}=\frac{1}{4t\|V\|_1}+O(t^{-2}).
\ee
The leading term above will cancel the boundary term that arose in Proposition~\ref{freeevol}.

The  decay rate $\f{1}{t\log^2(t)}$ appears because of the following lemma, which seems to be optimal.
Define 
$$k(x,x_1):=1+\log^-(|x-x_1|)+\log^+(|x_1|),$$ 
where $\log^-(x)=|\log(x)|\chi_{(0,1)}(x)$ and $\log^+(x)= \log(x) \chi_{(1,\infty)}(x)$.
\begin{lemma}\label{lem:Sll2} For $t>2$, we have the bound
$$
|\eqref{Sll2}| \les \frac{1}{t\log^2(t)} k(x,x_1)  k(y,y_1)  \sqrt{w(x)  w(y) }.
$$ 
\end{lemma}
\begin{lemma}\label{lem:Sll3} Let $0<\alpha<1/4$. For $t>2$, we have the bound
$$
|\eqref{Sll3}|\les t^{-1-\alpha} k(x,x_1)  k(y,y_1)  \big(\la x\ra \la y\ra \la x_1\ra \la y_1\ra\big)^{\f12+\alpha+}.
$$
\end{lemma}
We will prove  Lemma~\ref{lem:Sll2} and Lemma~\ref{lem:Sll3} after we finish the proof of the proposition.

Using the bounds we obtained in \eqref{lem:Sll1}, Lemma~\ref{lem:Sll2},   Lemma~\ref{lem:Sll3} in \eqref{stone2_2}, we obtain
\begin{align*}
\eqref{stone2_2} = &\frac{1}{4t\|V\|_1} \int_{\R^4} v(x_1)S(x_1,y_1)v(y_1) dx_1 dy_1\\
&+O\Big( \f{\sqrt{w(x)w(y)} }{t\log^2(t)}   \int_{\R^4}  k(x,x_1) v(x_1)|S(x_1,y_1)| v(y_1)k(y,y_1)   dx_1 dy_1   \Big) \\
  &+ O\Big( \f{\big(\la x\ra \la y\ra\big)^{\f12+\alpha+}}{t^{1+\alpha} }    \int_{\R^4}
 k(x,x_1)  \la x_1\ra^{\f12+\alpha+}
 v(x_1)|S(x_1,y_1)| v(y_1) k(y,y_1)  ^{\f12+\alpha+}    dx_1 dy_1   \Big).
 \end{align*}
 Note that the integrals in the error terms are bounded in $x,y$, since 
$$\|v(y_1)\la y_1\ra^{\f12+\alpha+} k(y,y_1)\|_{L^2_{y_1}}\les 1.$$
 Also note that we can replace $S$ with $P$ in the first integral since the other parts of the operator $S$ contains $Q$ on at least one side and that $Qv=0$. Therefore,
\begin{multline*}
\eqref{stone2_2}= \frac{1}{4t\|V\|_1} \int_{\R^4} v(x_1)P(x_1,y_1)v(y_1) dx_1 dy_1+ O\Big( \f{\sqrt{w(x)w(y)}  }{t\log^2(t)} \Big)+ O\Big( \f{\big(\la x\ra \la y\ra\big)^{\f32+}}{t^{\f54} } \Big)\\ =\frac{1}{4t }+O\Big( \f{\sqrt{w(x)w(y)}  }{t\log^2(t)} \Big)+O\Big( \f{\big(\la x\ra \la y\ra\big)^{\f32+}}{t^{\f54} } \Big).
\end{multline*}
\end{proof}

\begin{proof}[Proof of Lemma~\ref{lem:Sll2}]
First note that
\be\label{G0bound}
|\widetilde G_0(x,x_1)|\les 1+|\log|x-x_1||\les k(x,x_1) \sqrt{w(x)}.
\ee
Second, we bound the $\lambda$-integral by using Lemma~\ref{lem:ibp} with $\mathcal E(\lambda)=\frac{\chi(\lambda)}{(\log(\lambda)+c_1)^2+c_2^2}$.
Note that
\begin{align*}
|\partial_\lambda \mathcal E(\lambda)|  \les   \frac{\chi(\lambda)}{\lambda |\log(\lambda)|^3}, \quad\quad \Big|\partial_\lambda\Big(\frac{\partial_\lambda \mathcal E (\lambda)}{\lambda}\Big)\Big|  \les  \frac{\chi(\lambda)}{\lambda^3 |\log(\lambda)|^3}.
 \end{align*}
Applying Lemma~\ref{lem:ibp} with these bounds we obtain
\begin{multline*}
\Big|\int_0^\infty e^{it\lambda^2} \lambda \,  \mathcal E(\lambda) \, d\lambda \Big| \les \f1t\int_0^{t^{-1/2}}|\mathcal E^\prime(\lambda)| d\lambda+ \Big|\frac{\mathcal{E}^\prime(t^{-1/2})}{t^{3/2}}\Big|
+\f1{t^2}\int_{t^{-1/2}}^\infty \Big|\Big(\frac{\mathcal E^\prime(\lambda)}{\lambda}\Big)^\prime\Big| d\lambda\\
\les \f{ 1}{t}\int_0^{t^{-1/2}} \frac{\chi(\lambda)}{\lambda |\log(\lambda)|^3} d\lambda +\f{1 }{t\log^3(t)}
+\f{1}{t^2}\int_{t^{-1/2}}^\infty\frac{\chi(\lambda)}{\lambda^3 |\log(\lambda)|^3}  d\lambda.
\end{multline*}
It is easy to calculate that
$$
 \f1t \int_0^{t^{-1/2}} \frac{1}{\lambda |\log(\lambda)|^3} d\lambda \sim \f1{t\log^2(t)}.
$$
It remains to bound the integral on $[t^{-1/2},\infty)$:
\begin{multline*}
\f{1}{t^2}\int_{t^{-1/2}}^\infty\frac{\chi(\lambda)}{\lambda^3 |\log(\lambda)|^3}  d\lambda \lesssim  \f1{t^2}+ \f1{t^2} \int_{t^{-1/2}}^{1/10}  \frac{1}{\lambda^3  |\log(\lambda)|^3 } d\lambda \\  \les \f1{t^2}+\f1{t^2} \int_{t^{-1/4}}^{1/10}  \frac{1}{\lambda^3  } d\lambda+\f1{t^2} \int_{t^{-1/2}}^{t^{-1/4}}  \frac{1}{\lambda^3  |\log(t)|^3 } d\lambda
\les \f1{t^{3/2}}+\f1{t|\log(t)|^3}.
\end{multline*}
The first inequality follows since the integral on $[\f1{10},\infty)$ converges.

Combining the bounds we obtained above finishes the proof of the lemma.
\end{proof}

Before we prove Lemma~\ref{lem:Sll3}, we discuss the following variant of Lemma~\ref{lem:ibp}:
\begin{lemma}\label{lem:ibp2} Assume that  $\mathcal E(0)=0$. For $t>2$, we have
\begin{multline}\label{ibp2}
\Big|\int_0^\infty e^{it\lambda^2} \lambda \, \mathcal E(\lambda) d\lambda  \Big| \les \f1t\int_0^{\infty}\frac{|\mathcal E^\prime(\sqrt s)|}{\sqrt s (1+st)} ds
+\f1{t}\int_{\f{\pi}{t}}^\infty \Big| \frac{\mathcal E^\prime(\sqrt{s+\f\pi{t}})-\mathcal E^\prime(\sqrt{s})}{\sqrt s} \Big| ds\\
\les \f1t\int_0^{\infty}\frac{|\mathcal E^\prime(\lambda)|}{  (1+\lambda^2 t)} d\lambda
+\f1{t}\int_{t^{-1/2}}^\infty \big|  \mathcal E^\prime(\lambda \sqrt{1+ \pi t^{-1}\lambda^{-2}} )-\mathcal E^\prime(\lambda)  \big| d\lambda.
\end{multline}
\end{lemma}
\begin{proof}
As before we  integrate by parts once using the identity $e^{it\lambda^2}\lambda=\partial_\lambda e^{it\lambda^2}/(2it)$, and then let $s=\lambda^2$ to obtain
$$
 \int_0^\infty e^{it\lambda^2} \lambda \mathcal E(\lambda) d\lambda =  \f{i}{2t}  \int_0^\infty e^{it\lambda^2}  \mathcal E^\prime(\lambda) d\lambda
 =\f{i}{4t}\int_0^\infty e^{its}  \frac{\mathcal E^\prime(\sqrt{s})}{\sqrt s} ds=\int_0^{\f{2\pi}{t}}+\int_{\f{2\pi}{t}}^\infty.
$$
The contribution of the first integral is bounded by the first integral on the right hand side of  \eqref{ibp2}.
We rewrite the second integral as
$$
\int_{\f{2\pi}{t}}^\infty e^{its}  \frac{\mathcal E^\prime(\sqrt{s})}{\sqrt s} ds = - \int_{\f{2\pi}{t}}^\infty e^{it(s-\f\pi{t})}  \frac{\mathcal E^\prime(\sqrt{s})}{\sqrt s} ds =  - \int_{\f{\pi}{t}}^\infty e^{its}  \frac{\mathcal E^\prime(\sqrt{s+\f\pi{t}})}{\sqrt{s+\f\pi{t}}} ds.
$$
Therefore it suffices to consider (the integral on $[\pi/t,2\pi/t]$ is bounded by the first integral on the right hand side of \eqref{ibp2})
$$
\int_{\f{\pi}{t}}^\infty e^{its}  \Big(\frac{\mathcal E^\prime(\sqrt{s})}{\sqrt s} - \frac{\mathcal E^\prime(\sqrt{s+\f\pi{t}})}{\sqrt{s+\f\pi{t}}}\Big) ds.
$$
The claim follows from
\begin{multline*}
\Big|\frac{\mathcal E^\prime(\sqrt{s})}{\sqrt s} - \frac{\mathcal E^\prime(\sqrt{s+\f\pi{t}})}{\sqrt{s+\f\pi{t}}}\Big| \les \f{|\mathcal E^\prime(\sqrt{s})-\mathcal E^\prime(\sqrt{s+\f\pi{t}})|}{\sqrt{s+\f\pi{t}}}+|\mathcal E^\prime(\sqrt{s})|\Big|\frac{1}{\sqrt s}-\f1{\sqrt{s+\f\pi{t}}}\Big|\\
\les  \f{|\mathcal E^\prime(\sqrt{s})-\mathcal E^\prime(\sqrt{s+\f\pi{t}})|}{\sqrt{s}}+  \frac{|\mathcal E^\prime(\sqrt{s})|}{ts^{\f32}}.
\end{multline*}
\end{proof}

\begin{proof}[Proof of Lemma~\ref{lem:Sll3}]
We will only consider the following part of \eqref{Sll3}:
\be\label{gecic}
\int_0^\infty e^{it\lambda^2} \lambda \chi(\lambda) \Big(1+\f{\widetilde G_0(x,x_1)}{g(\lambda)+c}\Big) E_0(\lambda)(y,y_1) d\lambda=: \int_0^\infty e^{it\lambda^2} \lambda \, \mathcal E(\lambda) \,d\lambda.
\ee
The other parts are either of this   form or
much smaller. We also omit the $\pm$ signs since we can not rely on a cancellation between '+' and '-' terms.

Using Lemma~\ref{R0 exp cor}, Corollary~\ref{lipbound},  and \eqref{G0bound},  we estimate (for $0<\lambda<b\les \lambda<\lambda_1$)
$$
 |\partial_\lambda\mathcal E(\lambda)|\les k(x,x_1) \sqrt{w(x)} \chi(\lambda) \lambda^{-\f12 } \la y-y_1\ra^{\f12 }\les  k(x,x_1) \sqrt{w(x) \la y\ra \la y_1\ra }\lambda^{-\f12 },
$$
\begin{multline*}
\big| \partial_\lambda \mathcal E (b)-\partial_\lambda \mathcal E(\lambda) \big|\les \chi(\lambda) k(x,x_1) \sqrt{w(x)} \lambda^{-\f12-\alpha} (b-\lambda)^\alpha \la y-y_1\ra^{\f12+\alpha } \\
\les \chi(\lambda) k(x,x_1) \sqrt{w(x)}   (\la y\ra \la y_1\ra)^{\f12+\alpha }  \lambda^{-\f12-\alpha} (b-\lambda)^\alpha.
\end{multline*}

Noting that $\mathcal E(0)=0$ we can use Lemma~\ref{lem:ibp2} to obtain
$$
|\eqref{gecic}|\les  \f1t\int_0^{\infty}\frac{|\mathcal E^\prime(\lambda)|}{  (1+\lambda^2 t)} d\lambda
+\f1{t}\int_{t^{-1/2}}^\infty \big|  \mathcal E^\prime(\lambda \sqrt{1+ \pi t^{-1}\lambda^{-2}} )-\mathcal E^\prime(\lambda)  \big| d\lambda.$$
Using the bounds above, we estimate the first integral by
$$
 \f{k(x,x_1) \sqrt{w(x) \la y\ra \la y_1\ra } }{t}\int_0^{\infty}\frac{1}{ \sqrt{\lambda} (1+\lambda^2 t)} d\lambda\les  \f{k(x,x_1) \sqrt{w(x) \la y\ra \la y_1\ra } }{t^{5/4}}.
$$
To estimate the second integral, we apply the Lipschitz bound with
$$b-\lambda=\lambda \big(\sqrt{1+ \pi t^{-1}\lambda^{-2}} -1\big)\sim \f1{t\lambda},$$
and get
$$
\f{ k(x,x_1) \sqrt{w(x)}   (\la y\ra \la y_1\ra)^{\f12+\alpha }  }{t}\int_{t^{-1/2}}^{\lambda_1}  \lambda^{-\f12-\alpha} (t\lambda)^{-\alpha} d\lambda
\les \f{ k(x,x_1) \sqrt{w(x)}   (\la y\ra \la y_1\ra)^{\f12+\alpha }  }{t^{1+\alpha}},
$$
since $\alpha \in (0,1/4)$.

Taking into account the contribution of the term with the roles of $x$ and $y$ switched, we obtain the assertion of the lemma.
\end{proof}

Next we consider the contribution of the third term in \eqref{resolvent id} to \eqref{stone2}:
\be\label{stone2_3}
 \int_{\R^4}\int_0^\infty e^{it\lambda^2}\lambda \chi(\lambda)  [\mathcal R^-_2 -\mathcal R_2^+] v(x_1)[QD_0Q](x_1,y_1)v(y_1)
d\lambda dx_1 dy_1,
\ee
where
\be\label{calR2}
\mathcal R^\pm_2= R_0^{\pm}(\lambda^2)(x,x_1) R_0^{\pm}(\lambda^2)(y_1,y).
\ee
Recall from  Lemma~\ref{R0 exp cor}  that
	\begin{align*}
		R_0^{\pm}(\lambda^2)(x,x_1)=c[a\log(\lambda|x-x_1|)+b\pm  i]+ E_0^\pm(\lambda)(x,x_1),
	\end{align*}
where $a,b,c\in\R$.  Therefore
\begin{multline*}
\mathcal R_2^\pm=c^2\big[(a\log(\lambda|x-x_1|)+b)(a\log(\lambda|y-y_1|)+b)-1 \big]\\  \pm i c^2 \big[ a\log(\lambda|x-x_1|)+a\log(\lambda|y-y_1|)+2b  \big] +E_3^\pm(\lambda),
\end{multline*}
where
\begin{multline}\label{E3def}
E_3^\pm(\lambda):= c[a\log(\lambda|x-x_1|)+b\pm  i] E_0^\pm(\lambda)(y,y_1)\\ +
c[a\log(\lambda|y-y_1|)+b\pm  i] E_0^\pm(\lambda)(x,x_1)+ E_0^\pm(\lambda)(x,x_1)E_0^\pm(\lambda)(y,y_1).
\end{multline}
Using this, we have
$$
\mathcal R_2^--\mathcal R_2^+=-2c^2(a\log(\lambda|x-x_1|)+a\log(\lambda|y-y_1|)+2b)+E_3^-(\lambda)-E^+_3(\lambda).
$$
Using this in \eqref{stone2_3}, and noting that the contribution of the first summand vanishes since $Qv=0$, we obtain
\be\label{stone2_3_2}
\eqref{stone2_3} = \int_{\R^4}\int_0^\infty e^{it\lambda^2}\lambda \chi(\lambda)  [E_3^-(\lambda)-E^+_3(\lambda)] v(x_1)[QD_0Q](x_1,y_1)v(y_1)
d\lambda dx_1 dy_1.
\ee

\begin{prop}\label{prop:stone2_3} Let $0<\alpha<1/4$. If $v(x)\les \la x\ra^{-\f32-\alpha-}$, then
we have
$$
\eqref{stone2_3}=O\Big(\f{\la x\ra^{\f12+\alpha+} \la y \ra^{\f12+\alpha+} }{t^{1+\alpha}}\Big).
$$
\end{prop}
\begin{proof}
Let $\mathcal E(\lambda)= \chi(\lambda) E_3(\lambda)$ (we dropped the '$\pm$' signs).
Using
$$
|\log|x-x_1||\les k(x,x_1) \sqrt{w(x)},
$$
and the bounds in Lemma~\ref{R0 exp cor} and Corollary~\ref{lipbound} we estimate
(for $0<\lambda<b\les \lambda<\lambda_1$)
$$
|\partial_\lambda\mathcal E(\lambda)| \les \chi(\lambda) \lambda^{-\f12-} (\la y \ra\la x\ra \la y_1\ra \la x_1\ra)^{\f12+}  k(x,x_1)k(y,y_1),
$$
$$
\big| \partial_\lambda \mathcal E (b)-\partial_\lambda \mathcal E(\lambda) \big|\les
 \chi(\lambda) k(x,x_1) k(y,y_1)   (\la x\ra\la x_1\ra\la y\ra \la y_1\ra)^{\f12+\alpha+ }  \lambda^{-\f12-\alpha-} (b-\lambda)^\alpha.
$$
Applying Lemma~\ref{lem:ibp2} together with these bounds as in the proof of the previous lemma, we bound the    $\lambda$-integral by
$$
  k(x,x_1) k(y,y_1)   (\la x\ra\la x_1\ra\la y\ra \la y_1\ra)^{\f12+\alpha+ }\f{1}{t^{1+\alpha}}.
$$
Therefore,
\begin{multline*}
\eqref{stone2_3}
\les t^{-1-\alpha} \int_{\R^4} k(x,x_1)k(y,y_1) (\la y \ra \la y_1\ra \la x\ra \la x_1\ra )^{\f12+\alpha+} v(x_1)|QD_0Q(x_1,y_1)|v(y_1) dx_1dy_1
\\ \les \f{\la x\ra^{\f12+\alpha+} \la y \ra^{\f12+\alpha+}}{t^{1+\alpha}},
\end{multline*}
since  $\|v(x_1)k(x,x_1) \la x_1\ra^{\f12+\alpha+}\|_{L^2_{x_1}}\les 1$.
\end{proof}

We now turn to the contribution of the error term $E^{\pm}(\lambda)$ from Lemma~\ref{Minverse}
in \eqref{resolvent id}. Dropping the '$\pm$' signs, we need to consider
\be\label{stone2_er}
    		  \int_{\R^4}  \int_0^\infty e^{it\lambda^2} \lambda \, \mathcal E (\lambda)  v(x_1) v(y_1)  \, d\lambda
		  \, dx_1\, dy_1,
  	\ee
where
$$
\mathcal E (\lambda):= \chi(\lambda)
		R_0(\lambda^2)(x,x_1)E(\lambda)(x_1,y_1) R_0(\lambda^2)(y,y_1).
$$
\begin{prop} \label{prop:stone2er} Let $0<\alpha<1/4$. If $v(x)\les \la x\ra^{-\f32-\alpha-}$, then
we have
  	$$
    	\eqref{stone2_er}= O\Big(\f{\la x\ra^{\f12+\alpha+} \la y \ra^{\f12+\alpha+}}{t^{1+\alpha}}\Big).
  	$$
\end{prop}
\begin{proof}
Let
\begin{multline*}
		T_0:= \sup_{0<\lambda<\lambda_1} \lambda^{-\frac{1}{2} } |E^{\pm}(\lambda)|+\sup_{0<\lambda<\lambda_1} \lambda^{\frac{1}{2} } |\partial_\lambda E^{\pm}(\lambda)| \\ + \sup_{0<\lambda<b\les\lambda<\lambda_1}  \f{\lambda^{\frac{1}{2}+\alpha}}{ (b-\lambda)^{\alpha}} |\partial_\lambda E^{\pm}(b)-\partial_\lambda E^\pm(\lambda)|.
\end{multline*}

By Lemma~\ref{Minverse}, we see that $T_0 $ is  Hilbert-Schmidt on $L^2(\R^2)$, and hence we have the following bounds for the kernels
$$
|E^{\pm}(\lambda)|\les \lambda^{\f12} T_0,\,\,\,\,|\partial_\lambda E^{\pm}(\lambda)| \les \lambda^{-\f12} T_0,
$$
$$
|\partial_\lambda E^{\pm}(b)-\partial_\lambda E^\pm(\lambda)|\les \lambda^{-\frac{1}{2}-\alpha}(b-\lambda)^{\alpha} T_0,\,\,\,\,\,\,\text{ if } 0<\lambda<b\les \lambda<\lambda_1.
$$
Moreover, using Lemma~\ref{R0 exp cor}  and Corollary~\ref{lipbound}, we have (for $0<\lambda<b\les \lambda<\lambda_1$)
\begin{align*}
&|R_0(\lambda^2)(x,x_1)| \les (1+ |\log \lambda|) k(x,x_1) \sqrt{w(x)}\les \lambda^{0-}k(x,x_1) \la x\ra^{0+},\\
&|\partial_\lambda R_0(\lambda^2)(x,x_1)|\les \f1\lambda + \lambda^{-\f12} \sqrt{ \la x\ra \la x_1\ra}, \\
&|\partial_\lambda R_0(\lambda^2)(x,x_1)-\partial_\lambda R_0(b^2)(x,x_1)|\les (b-\lambda)^\alpha \Big[\f1{\lambda^{1+\alpha}} + \f{|x-x_1|^{\f12+\alpha}}{\lambda^{\f12}}\Big].
\end{align*}
Therefore we have the bounds (for $0<\lambda<b\les \lambda<\lambda_1$)
\begin{align*}
&|\partial_\lambda\mathcal E(\lambda)| \les  \lambda^{-\f12-} (\la y \ra \la x\ra \la y_1\ra \la x_1\ra)^{\f12}  k(x,x_1)k(y,y_1)T_0(x_1,y_1),\\
&|\partial_\lambda \mathcal E (b)-\partial_\lambda \mathcal E (\lambda)|\les \lambda^{-\frac{1}{2}-\alpha-}(b-\lambda)^{\alpha} (\la y \ra \la x\ra \la y_1\ra \la x_1\ra)^{\f12+\alpha+}  k(x,x_1)k(y,y_1)T_0(x_1,y_1).
\end{align*}
Applying Lemma~\ref{lem:ibp2} as above yields the claim of the proposition.
\end{proof}
Note that Proposition~\ref{freeevol}, Proposition~\ref{prop:stone2_2}, Proposition~\ref{prop:stone2_3}, and Proposition~\ref{prop:stone2er} yield Theorem~\ref{thm:main}.

\section{Proof of Theorem~\ref{thm:mainineq} For Energies Away From Zero}

In this section we prove Theorem~\ref{thm:mainineq} for energies separated from zero:
\begin{theorem} \label{thm:mainineq high}
	Under the assumptions of Theorem~\ref{thm:main}, we have
	for $t>2$
	\be\label{weighteddecayhi}
		\sup_{L\geq1}\bigg| \int_0^\infty e^{it\lambda^2}\lambda \widetilde{\chi}(\lambda) \chi(\lambda/L)
		 [R_V^+(\lambda^2)-R_V^-(\lambda^2)](x,y) d\lambda \bigg|
		 \les\f{\la x\ra^{\f32}\la y\ra^{\f32}}{t^{\f32}}
	\ee
where $\widetilde\chi=1-\chi$.
\end{theorem}
\begin{proof}
We start with the resolvent expansion
\begin{align}
    R_V^{\pm}(\lambda^2)&=\sum_{m=0}^{2M+2} R_0^{\pm}(\lambda^2)(-VR_0^{\pm}(\lambda^2))^m \label{born series}\\
    &+R_0^{\pm}(\lambda^2)(VR_0^{\pm}(\lambda^2))^M VR_V^{\pm}(\lambda^2) V (R_0^{\pm}(\lambda^2)V)^M R_0^{\pm}(\lambda^2).
    \label{born tail}
\end{align}
We first note that the contribution of the term $m=0$ can be handled as in Proposition~\ref{freeevol} and it can be bounded by 
$\frac{\la x\ra^{\f32}\la y\ra^{\f32}}{t^2}$. For the case $m>0$ we won't make use of any  cancellation between `$\pm$' terms. Thus, we will only consider $R_0^-$, and drop the `$\pm $' signs.
Using \eqref{R0 def}, \eqref{J0 def}, \eqref{Y0 def}, and \eqref{JYasymp2} we write
\begin{align*}
    R_0(\lambda^2)(x,y)=e^{-i\lambda|x-y|}\rho_+(\lambda|x-y|)+\rho_-(\lambda|x-y|),
\end{align*}
where $\rho_+$ and $\rho_-$ are supported on the sets $[1/4,\infty)$ and $[0, 1/2]$, respectively.
Moreover, we have the bounds
\be\label{omega bds}
    \rho_-(y) =\widetilde O(1+|\log y|),\,\,\,\,\,\,\, \rho_+(y) =\widetilde O\big( (1+|y|)^{-1/2}\big)\\
\ee

We first control the contribution of the finite born series, \eqref{born series}, for $m>0$.
Note that the contribution of the $m$th term of \eqref{born series} to the integral in \eqref{weighteddecayhi} can be written as a sum of integrals of the form
\be 
		\int_{\R^{2m}} \int_0^\infty e^{it\lambda^2}\lambda \widetilde \chi(\lambda)
		\chi(\lambda/L) e^{- i \lambda \sum_{j\in J}d_j} \prod_{j\in J} \rho_+
		(\lambda d_j)
		\prod_{\ell\in J^*} \rho_-(\lambda d_\ell) 
		\prod_{n=1}^{m} V(x_n) \, d\lambda \, dx_1\dots dx_{m},\label{high born1}
\ee
where $d_j=|x_{j-1}-x_j|$ and $J\cup J^*$ is a partition of $\{1,...,m,m+1\}.$ 
Let 
$$
\mathcal E(\lambda) := \widetilde \chi(\lambda)
		\chi(\lambda/L) e^{- i \lambda \sum_{j\in J}d_j} \prod_{j\in J} \rho_+
		(\lambda d_j)
		\prod_{\ell\in J^*} \rho_-(\lambda d_\ell).
$$
To estimate the derivatives of $\mathcal E$, we note that
\begin{align*}
\big|\partial_\lambda^k \big[\rho_+(\lambda d_j)\big]\big|&\les \frac{d_j^k}{(1+\lambda d_j)^{k+1/2}},\,\,\,\,\,\,\,k=0,1,2,...,\\
\big|\partial_\lambda^k \big[\rho_-(\lambda d_j)\big]\big|&\les \f1{\lambda^k},\,\,\,\,k=1,2,...
\end{align*}
Using the monotonicity of $\log^-$ function, we also obtain
$$
\widetilde\chi(\lambda) \big|\rho_-(\lambda d_j)\big|\les \widetilde \chi(\lambda) (1+|\log(\lambda d_j)|) \chi_{\{0<\lambda d_j\leq 1/2\}}
\les \widetilde \chi(\lambda) (1+\log^-(\lambda d_j))\les
1+\log^-(d_j).
$$
It is also easy to see that
$$
\Big|\frac{d^k}{d\lambda^k} \chi(\lambda/L)\Big|\les \lambda^{-k}.
$$
Finally, noting that
 $(\widetilde \chi)^{\prime}$ is supported on the set $\{\lambda\approx 1\}$, we can estimate
\begin{multline} \label{ehiprime}
\big|\partial_\lambda \mathcal E\big| \les \widetilde \chi(\lambda) \Big(\f1\lambda +\sum_{k\in J}\big(d_k + \frac{d_k}{1+\lambda d_k} \big)\Big)
\prod_{j\in J} \frac{1}{(1+\lambda d_j)^{1/2}}
		\prod_{\ell\in J^*} (1+\log^-(d_\ell))\\
\les  \widetilde \chi(\lambda) \Big(\f1\lambda+\sum_{k\in J} \frac{d_k}{(1+\lambda d_k)^{1/2}}  \Big)
		\prod_{\ell\in J^*} (1+\log^-(d_\ell))
\les  \widetilde \chi(\lambda) \Big(\lambda^{-1}+\sum_{k\in J} d_k^{\f12} \lambda^{-\f12}  \Big)
		\prod_{\ell\in J^*} (1+\log^-(d_\ell))\\
\les \widetilde \chi(\lambda) \lambda^{-\f12} \prod_{k=0}^{m+1} \la x_k\ra^{\f12}  
		\prod_{\ell=1}^{m+1} (1+\log^-(d_\ell)).
\end{multline}
We also have
\begin{multline} \label{ehiprime2}
\big|\partial_\lambda^2 \mathcal E\big| \les \widetilde \chi(\lambda) \Big(\f1{\lambda^2}+\sum_{k\in J}\big(d_k^2 + \frac{d_k^2}{(1+\lambda d_k)^2}\big)   \Big)
\prod_{j\in J} \frac{1}{(1+\lambda d_j)^{1/2}}
		\prod_{\ell\in J^*} (1+\log^-(d_\ell))\\
\les  \widetilde \chi(\lambda) \Big(\lambda^{-2}+\sum_{k\in J} d_k^{\f32} \lambda^{-\f12}   \Big)
		\prod_{\ell\in J^*} (1+\log^-(d_\ell))\les \widetilde \chi(\lambda) \lambda^{-\f12} \prod_{k=0}^{m+1} \la x_k\ra^{\f32}
		\prod_{\ell=1}^{m+1} (1+\log^-(d_\ell)).
\end{multline} 
Using Lemma~\ref{lem:ibp2} (and taking the support condition of $\widetilde \chi$ into account), we can bound the $\lambda$ integral in \eqref{high born1} by
\be\label{hilip}
\f1{t^2}\int_0^{\infty}\frac{|\mathcal E^\prime(\lambda)|}{ \lambda^2 } d\lambda
+\f1{t}\int_{0}^\infty \big|  \mathcal E^\prime(\lambda \sqrt{1+ \pi t^{-1}\lambda^{-2}} )-\mathcal E^\prime(\lambda)  \big| d\lambda,
\ee
Using \eqref{ehiprime}, we can bound the first integral in \eqref{hilip} by
\be\label{hilip1}
 \prod_{k=0}^{m+1} \la x_k \ra^{\f12} \prod_{\ell=1}^{m+1} (1+\log^-(d_\ell))\int_{0}^\infty \widetilde \chi(\lambda) \lambda^{-5/2} d\lambda\les \prod_{k=0}^{m+1} \la x_k \ra^{\f12} \prod_{\ell=1}^{m+1} (1+\log^-(d_\ell)).
\ee
To estimate the second integral in \eqref{hilip} first note that
\be\label{bminusa}
\lambda \sqrt{1+ \pi t^{-1}\lambda^{-2}}  - \lambda \approx \f1{t\lambda}.
\ee
Next  using \eqref{bminusa}, \eqref{ehiprime} and \eqref{ehiprime2}, we have (for any $0\leq \alpha \leq 1 $)
\begin{multline} \label{ehidiff}
 \big|  \mathcal E^\prime(\lambda \sqrt{1+ \pi t^{-1}\lambda^{-2}} )-\mathcal E^\prime(\lambda)  \big| \\ \les 
 \widetilde\chi(2\lambda) \lambda^{-\f12} \prod_{k=0}^{m+1} \la x_k\ra^{\f12}
		\prod_{\ell=1}^{m+1} (1+\log^-(d_\ell)) \,
 \min\Big(1, \f1{t\lambda} \prod_{k=0}^{m+1} \la x_k\ra\Big)\\
\les t^{-\alpha} \widetilde\chi(2\lambda) \lambda^{-\f12-\alpha} \prod_{k=0}^{m+1} \la x_k\ra^{\f12+\alpha}
		\prod_{\ell=1}^{m+1} (1+\log^-(d_\ell)).  
\end{multline}
Using this bound for $\alpha\in(1/2,1]$, we bound the second integral in \eqref{hilip} by
\begin{multline} \label{hilip2}
t^{-\alpha} \prod_{k=0}^{m+1} \la x_k\ra^{\f12+\alpha}
		\prod_{\ell=1}^{m+1} (1+\log^-(d_\ell)) \int_0^\infty \widetilde\chi(2\lambda) \lambda^{-\f12-\alpha} \les  \\ \les t^{-\alpha} \prod_{k=0}^{m+1} \la x_k\ra^{\f12+\alpha}
		\prod_{\ell=1}^{m+1} (1+\log^-(d_\ell)).
\end{multline}
Combining \eqref{hilip1} and \eqref{hilip2}, we obtain  
$$
|\eqref{hilip}|\les  t^{-1-\alpha} \prod_{k=0}^{m+1} \la x_k\ra^{\f12+\alpha}
		\prod_{\ell=1}^{m+1} (1+\log^-(d_\ell))
$$
Using this (with $\f12<\alpha<2\beta-\f52$) in \eqref{high born1}, we obtain 
\begin{multline*}
|\eqref{high born1}|\les  t^{-1-\alpha}
		\int_{\R^{2m}}  \prod_{k=0}^{m+1} \la x_k\ra^{\f12+\alpha}
		\prod_{\ell=1}^{m+1} (1+\log^-(d_\ell))
		\prod_{n=1}^{m} |V(x_n)|   \, dx_1\dots dx_{m}\\ \les \f{\la x_0\ra^{\f12+\alpha}\la x_{m+1}\ra^{\f12+\alpha}}{t^{\f32}}.
\end{multline*}

To control the remainder of the born series, \eqref{born tail}, we employ the limiting absorption principle,
see \cite{agmon},
\begin{align}\label{lap}
	\|\partial_\lambda^k R_V^{\pm}(\lambda^2)\|_{L^{2,\sigma}(\R^2)\to L^{2,-\sigma}(\R^2)}<\infty,
\end{align}
for $k=0,1,2$ with $\sigma>k+\frac{1}{2}$.  Similar bounds hold for the derivatives of the free resolvent.
In addition, for the free resolvent one has
\begin{align}\label{free lap}
	\|R_0^{\pm}(\lambda^2)\|_{L^{2,\sigma}(\R^2)\to L^{2,-\sigma}(\R^2)}\les \lambda^{-1+},
\end{align}
which is valid for $\sigma>\frac{1}{2}$.  Using the representation \eqref{omega bds},
we note the following bounds on the free resolvent
which are valid on $\lambda>\lambda_1>0$,
\begin{align*}
	|\partial_\lambda^k R_0^{\pm}(\lambda^2)(x,y)|\les |x-y|^k \left\{\begin{array}{ll}
	|\log(\lambda|x-y|)| & 0<\lambda|x-y|<\frac{1}{2}\\
	(\lambda|x-y|)^{-\frac{1}{2}} & \lambda|x-y|\gtrsim 1\end{array}\right.
	\les \lambda^{-\frac{1}{2}}|x-y|^{k-\frac{1}{2}}.
\end{align*}
Thus, for $\sigma>\f12+k$,
\begin{align}\label{R0 wtdL2}
	\|\partial_\lambda^k &R_0^{\pm}(\lambda^2)(x,y)\la y\ra^{-\sigma}\|_{L^2_y} 
	\les\lambda^{-\f12} \Big[\int_{\R^2} \frac{|x-y|^{2k-1}}{\la y\ra^{2\sigma}}\, dy\Big]^{\f12}
	\les \lambda^{-\f12} \la x\ra^{\max(0,k-1/2)}.
\end{align}
Once again, we estimate the $R_V^+$ and $R_V^-$ terms separately and omit the `$\pm$' signs.

	We  write the contribution of  \eqref{born tail} to \eqref{weighteddecayhi} as
	\begin{align}\label{I def}
		 \int_0^\infty e^{it\lambda^2} \lambda \, \mathcal E(\lambda)(x,y)\, d\lambda,
	\end{align}
where
$$
\mathcal E(\lambda)(x,y) = \widetilde{\chi}(\lambda) \chi(\lambda/L)	\big\la VR_V^{\pm}(\lambda^2) V (R_0^{\pm}(\lambda^2)V)^M R_0^{\pm}(\lambda^2)(\cdot,x), (R_0^{\pm}(\lambda^2)V)^M R_0^{\pm}(\lambda^2)(\cdot,y) \big\ra.
$$
Using \eqref{lap}, \eqref{free lap}, and \eqref{R0 wtdL2} (provided that $M\geq 2$) we see that 
	\begin{align}\label{a deriv bds}
		\big|\partial_\lambda^k \mathcal E(\lambda)(x,y)\big|
		&\les \widetilde{\chi}(\lambda) \chi(\lambda/L) \la \lambda \ra^{-2-}\la x\ra^{\f32}\la y\ra^{\f32}, \qquad k=0,1,2.
	\end{align}
	This requires that $|V(x)|\les \la x\ra^{-3-}$. One can see that the requirement on the decay rate of the potential arises when, for instance,
	both $\lambda$ derivatives act on one resolvent, this twice differentiated resolvent operator maps
	$L^{2,\frac{5}{2}+}\to L^{2,-\frac{5}{2}-}$ by \eqref{lap},  or is in $L^{2,-\frac{5}{2}-}$ by \eqref{R0 wtdL2}.
	The potential then needs to map
	$L^{2,-\frac{5}{2}-}\to L^{2,\frac{1}{2}+}$ for the next application of the limiting absorption principle.
	This is satisfied if $|V(x)|\les \la x\ra^{-3-}$.
	
	The required bound now follows by integrating by parts twice:
	\begin{align*}
		|\eqref{I def}|\les |t|^{-2} \int_0^\infty \bigg|\partial_\lambda
		\bigg(\frac{\partial_\lambda \,\mathcal E(\lambda)(x,y)}{\lambda}\bigg)\bigg|\, d\lambda
		\les |t|^{-2} \la x\ra^{\frac{3}{2}}\la y\ra^{\frac{3}{2}}.
	\end{align*}
	
\end{proof}

\begin{large}
\noindent
{\bf Acknowledgment. \\}
\end{large}
The authors would like to thank Wilhelm Schlag for suggesting this problem. The first author was partially supported by National Science Foundation grant  DMS-0900865.

\end{document}